\documentclass[reqno,11pt]{amsart}
\usepackage{amsmath,amssymb,latexsym,soul,cite,mathrsfs,accents}
\usepackage{color,enumitem,graphicx}
\usepackage[colorlinks=true,urlcolor=blue,
citecolor=red,linkcolor=blue,linktocpage,pdfpagelabels,
bookmarksnumbered,bookmarksopen]{hyperref}
\usepackage[english]{babel}
\usepackage[left=2.6cm,right=2.6cm,top=2.9cm,bottom=2.9cm]{geometry}
\usepackage{tensor}
\usepackage[abbrev]{amsrefs}
\setenumerate{label=(\roman*)}

\DeclareMathOperator{\Aut}{Aut}

\DeclareMathOperator{\Id}{Id}

\DeclareMathOperator{\Vol}{Vol}
\DeclareMathOperator{\dvol}{dV}

\DeclareMathOperator{\Ric}{Ric}

\DeclareMathOperator{\Conf}{Conf}

\newcommand{\defn}[1]{{\boldmath\bfseries#1}}

\newcommand{\grd}{g_{\mathrm{rd}}}
\newcommand{\ghyp}{g_{\mathrm{hyp}}}

\newcommand{\Diff}{\mathrm{Diff}}

\newcommand{\cg}{\widetilde{g}}

\newcommand{\cu}{\widetilde{u}}

\newcommand{\cM}{\widetilde{M}}

\newcommand{\cX}{\widetilde{X}}

\newcommand{\cpi}{\widetilde{\pi}}

\newcommand{\csigma}{\widetilde{\sigma}}

\newcommand{\hg}{\widehat{g}}

\newcommand{\hG}{\widehat{G}}

\newcommand{\lv}{\lvert}
\newcommand{\rv}{\rvert}

\newcommand{\mA}{\mathcal{A}}

\newcommand{\mF}{\mathcal{F}}

\newcommand{\mI}{\mathcal{I}}

\newcommand{\mM}{\mathcal{M}}

\newcommand{\mS}{\mathcal{S}}

\newcommand{\kD}{\mathfrak{D}}

\newcommand{\bN}{\mathbb{N}}

\newcommand{\bR}{\mathbb{R}}

\newcommand{\Sch}{\mathsf{P}}

\def\sideremark#1{\ifvmode\leavevmode\fi\vadjust{\vbox to0pt{\vss
 \hbox to 0pt{\hskip\hsize\hskip1em
 \vbox{\hsize3cm\tiny\raggedright\pretolerance10000
 \noindent #1\hfill}\hss}\vbox to8pt{\vfil}\vss}}}

\newcommand{\suchthat}{\mathrel{}:\mathrel{}}

\newtheorem{theorem}{Theorem}[section]
\newtheorem{proposition}[theorem]{Proposition}
\newtheorem{lemma}[theorem]{Lemma}

\theoremstyle{definition}
\newtheorem{definition}[theorem]{Definition}

\newcommand{\innerthmname}{}

\theoremstyle{definition}

\makeatletter
\def\namedlabel#1#2{\begingroup
	#2%
	\def\@currentlabel{#2}%
	\phantomsection\label{#1}\endgroup
}

\def\XXint#1#2#3{{\setbox0=\hbox{$#1{#2#3}{\int}$ }
		\vcenter{\hbox{$#2#3$ }}\kern-.6\wd0}}

\newcommand*\owedge{\mathpalette\@owedge\relax}
\newcommand*\@owedge[1]{%
	\mathbin{%
		\ooalign{%
			$#1\m@th\bigcirc$\cr
			\hidewidth$#1\m@th\wedge$\hidewidth\cr
		}%
	}%
}
\makeatother

\theoremstyle{remark}
\newtheorem{remark}[theorem]{Remark}

\numberwithin{equation}{section}

\title[A general nonuniqueness result for conformally variational invariants]{A general nonuniqueness result for Yamabe-type problems for conformally variational Riemannian invariants}  

\author[J.H. Andrade]{Jo\~{a}o H. Andrade}
\author[J.S. Case]{Jeffrey S. Case}
\author[P. Piccione]{Paolo Piccione}
\author[J. Wei]{Juncheng Wei}

\address[J.H. Andrade]{
	Institute of Mathematics and Statistics,
	University of S\~ao Paulo
	\newline\indent 
	S\~ao Paulo, SP 05508-090, Brazil}
\email{\href{mailto:andradejh@ime.usp.br}{andradejh@ime.usp.br}}

\address[J.S. Case]{
	Department of Mathematics \\
	The Pennsylvania State University
    \newline\indent
    University Park, PA 16802, USA}
\email{\href{mailto:jscase@psu.edu}{jscase@psu.edu}}

\address[P. Piccione]{
    Department of Mathematics, 
    School of Sciences, Great Bay University
    \newline\indent 
    523000, Dongguan-GD, People’s Republic of China
    \newline\indent 
    and
    \newline\indent
    School of Mathematical Sciences, Zhejiang Normal University
    \newline\indent 
    321004, Jinhua-ZJ, People’s Republic of China
    \newline\indent 
    and
    \newline\indent
    (Permanent address) Institute of Mathematics and Statistics,	University of S\~ao Paulo
    \newline\indent 
    05508-090, S\~ao Paulo-SP, Brazil}
    \email{\href{mailto:piccione@ime.usp.br}{paolo.piccione@usp.br}}

\address[J. Wei]{
	Department of Mathematics,
	The Chinese University of Hong Kong
	\newline\indent 
	Shatin-NT, Hong Kong}
\email{\href{mailto:wei@math.cuhk.edu.hk}{wei@math.cuhk.edu.hk}}

\subjclass[2020]{Primary 58J55; Secondary 35J60, 35B09, 35J30, 53C18}
\keywords{Conformally invariant equations, prescribed curvature, conformally variational invariants}

\begin{document}
	
	\begin{abstract}
		Given a conformally variational scalar Riemannian invariant $I$, we identify a sufficient condition for a compact Riemannian manifold to admit finite regular coverings with many nonhomothetic conformal rescalings with $I$ constant.
        We also identify a sufficient condition for the universal cover to admit infinitely many nonhomothetic periodic conformal rescalings with $I$ constant.
        Using these conditions, we improve known nonuniqueness results for the $Q$-curvatures of orders two, four, and six.
        We also prove nonuniqueness results for higher-order $Q$-curvatures and renormalized volume coefficients.
	\end{abstract}
	
	\maketitle

	\numberwithin{equation}{section} 
	\numberwithin{theorem}{section}
		
	\section{Introduction}\label{sec:intro}
 
The affirmative solution to the Yamabe Problem~\cite{MR888880} asserts that if $(M^n,g_0)$, $n \geq 3$, is a compact Riemannian manifold, then there is a unit volume conformal rescaling $g \in [g_0]$ such that
\begin{equation*}
 Q_2^g = Y_{Q_2}(M^n,[g_0]) := \inf_{\hg \in [g_0]} \left\{ \int_M Q_2^{\hg} \dvol_{\hg} \suchthat \Vol_{\hg}(M) = 1 \right\}
\end{equation*}
is constant, where $Q_2 := R/2(n-1)$ is proportional to the scalar curvature.
Bettiol and Piccione~\cite{MR3803113} used this fact to produce numerous examples of Riemannian manifolds that admit many nonhomothetic conformal rescalings with constant scalar curvature.
More precisely, if $(M^n,g)$, $n \geq 3$, is \emph{any} compact Riemannian manifold with positive Yamabe constant whose fundamental group has infinite profinite completion---for example, appropriately chosen Riemannian products with a compact symmetric space of Euclidean or negative type---then for each $\ell \in \bN$ there is a finite regular covering $\pi \colon \cM^n \to M^n$ for which the conformal manifold $(\cM^n,[\pi^\ast g])$ admits pairwise nonhomothetic metrics $\{ g_j \}_{j=1}^\ell \subset [\pi^\ast g]$ of constant scalar curvature.
This gives a far-reaching generalization of Schoen's observation~\cite{MR994021} that the number of conformal rescalings of the product metric on $S^1(L) \times S^{n-1}$ with constant scalar curvature tends to $\infty$ as $L \to \infty$.

Branson~\cite{MR1316845} introduced the $Q$-curvature $Q_{2k}$ of order $2k$ as a higher-order generalization of the scalar curvature.
Bettiol and Piccione's nonuniqueness result has been generalized to these when $k \leq 3$.
To run their argument, one needs a condition that is stable with respect to finite covers and guarantees the existence of a smooth minimizer of the Yamabe-type constant
\begin{equation*}
 Y_{Q_{2k}}^+(M^n,[g_0]) := \inf_{\hg \in [g_0]} \left\{ \int_M Q_{2k}^{\hg} \dvol_{\hg} \suchthat \Vol_{\hg}(M) = 1 \right\} .
\end{equation*}
Conditions of this type are known for orders at most six~\cites{MR3518237,MR888880,MR3420504,MR3509928,MR3652455} or under restrictive assumptions on the conformal manifold~\cites{MR2219215,case-malchiodi}.
A perturbation argument yields nonuniqueness results~\cite{arXiv:2302.11073} for the fractional $Q$-curvature~\cite{GonzalezQing2013} $Q_{2\gamma}$ when $\gamma$ is sufficiently close to $1/2$ or $1$.

By lifting the above metrics to the universal cover, one obtains noncompact manifolds with infinitely many periodic metrics of constant $Q_{2k}$-curvature.
Since the fundamental group of a compact hyperbolic manifold, and hence any product thereof, has infinite profinite completion~\cite{MR1299730}, this argument produces infinitely many solutions to the singular $Q_{2k}$-Yamabe Problem on $S^n \setminus S^\ell$ under appropriate hypotheses on $k$, $\ell$, and $n$ (cf.\ \cite{MR3803113}*{Corollary~1.2}).
However, since the isometry group of a compact manifold $(M^n,g)$ need not be isomorphic to the isometry group of its universal cover---for example, compact hyperbolic manifolds have finite isometry group, but the isometry group of the universal cover is noncompact---it is unclear how many of the lifted metrics are nonhomothetic.
In the companion paper~\cite{ACPW}, we use the Ferrand--Obata Theorem~\cites{MR1371767,MR0303464,MR1334876} to show that the lifted metrics are generically nonhomothetic in the case of the scalar curvature.

The purpose of this article is two-fold.
First, we develop an axiomatic framework which adapts the Bettiol--Piccione argument to variational scalar invariants.
This framework exposes the essential ingredients of their argument and is powerful enough to deduce that the periodic metrics in the universal cover are generically nonhomothetic.
Second, we demonstrate the utility of our axiomatic framework by indicating how it recovers known nonuniqueness results and using it to prove three new nonuniqueness results.
We begin with a discussion of the latter.

Our first two applications are to the nonuniqueness of metrics with constant $Q$-curvature on certain product manifolds.

\begin{theorem}
 \label{q-corollary}
 Fix $k,m \in \bN$.
 There is an $N = N(k,m) \in \bN$ such that if
 \begin{enumerate}
  \item[{\rm (i)}] $(M_1^m,g_1)$ is a compact Riemannian manifold for which $\Ric_{g_1} = -(m-1)g_1$ and $\pi_1(M_1)$ has infinite profinite completion, and
  \item[{\rm (ii)}] $(M_2^{n-m},g_2)$ is a compact Riemannian manifold for which $\Ric_{g_2} = (n-m-1)g_2$ and $n \geq N$,
 \end{enumerate}
 then for each $\ell \in \bN$, there is a finite regular covering $\pi \colon \cX^n \to M_1^m \times M_2^{n-m}$ for which there exist pairwise nonhomothetic metrics $\{ \sigma_j \}_{j=1}^\ell \subset [\pi^\ast (g_1 \oplus g_2)]$ of constant $Q_{2k}$-curvature.
 Moreover, the conformal universal cover $( \cM_1^m \times \cM_2^{n-m}, [\cpi^\ast(g_1 \oplus g_2)])$ admits infinitely many pairwise nonhomothetic periodic representatives $\{\csigma_j\}_{j=1}^\infty$ of constant $Q_{2k}$-curvature.
\end{theorem}

In general, the constant $N$ is hard to estimate~\cite{case-malchiodi}.
However, $N$ is known when restricting to locally conformally flat manifolds~\cite{MR2219215}:

\begin{theorem}
 \label{lcf-q-corollary}
 Fix $k,m \in \bN$.
 Let $n \geq 2k + 2m - 1$ be a positive integer.
 For every compact hyperbolic manifold $(M^m,\ghyp)$ and every $\ell \in \bN$, there is a finite regular covering $\pi \colon \cM^m \to M^m$ such that $(\cM^m \times S^{n-m}, [\pi^\ast\ghyp \oplus \grd])$ admits pairwise nonhomothetic representatives $\{ \sigma_j \}_{j=1}^\ell$ of constant $Q_{2k}$-curvature.
 Moreover, $(H^m \times S^{n-m}, [\ghyp \oplus \grd] )$ admits infinitely many pairwise nonhomothetic periodic conformal representatives of constant $Q_{2k}$-curvature.
\end{theorem}

Since $(H^m \times S^{n-m}, [\ghyp \oplus \grd])$ is conformally equivalent to $(S^n \setminus S^{m-1}, [\grd])$~\cite{MR1139641}, Theorem~\ref{lcf-q-corollary} gives a nonuniqueness result for singular (complete) metrics with constant $Q$-curvature.
In this form, Theorem~\ref{lcf-q-corollary} improves known results~\cites{MR994021,arXiv:2306.00679,MR4251294,MR3803113} when $k \leq 3$, recovers a result~\cite{ACPW} in our companion paper when $k=1$, and is entirely new for larger values of $k$.

Our third application is to the nonuniqueness of metrics with constant $\sigma_k$-curvature on particular locally conformally flat product manifolds.

\begin{theorem}
 \label{vk-corollary}
 Fix $k,m \in \bN$.
 There is an $N = N(k,m) \in \bN$ such that if $n \geq N$, then for every compact hyperbolic manifold $(M^m,\ghyp)$ and every $\ell \in \bN$, there is a finite regular covering $\pi \colon \cM^m \to M^m$ such that $(\cM^m \times S^{n-m}, [\pi^\ast \ghyp \oplus \grd])$ admits pairwise nonhomothetic representatives $\{ \sigma_j \}_{j=1}^\ell$ of constant $v_k$-curvature.
 Moreover, $(H^m \times S^{n-m}, [\ghyp \oplus \grd] )$ admits infinitely many pairwise nonhomothetic periodic representatives of constant $v_k$-curvature.
\end{theorem}

Our convention is that $S^1$ is a compact one-dimensional hyperbolic manifold, so Theorem~\ref{vk-corollary} significantly generalizes a nonuniqueness result of Viaclovsky~\cite{MR1738176}.
See Remark~\ref{rk:sigmak-N-values} for a discussion of the values $N(k,m)$.

The axiomatic framework underlying the above results is succinctly formulated in terms of conformally variational invariants, or CVIs~\cites{MR3955546,MR4392224}.
Roughly speaking, a CVI
of weight $-2k$ is a natural scalar Riemannian invariant $I$ such that
\begin{equation*}
 \left. \frac{\mathrm{d}}{\mathrm{d}t}\right|_{t=0} \int_M I^{e^{2tu}g} \dvol_{e^{2tu}g} = (n-2k)\int_M uI^g \dvol_g
\end{equation*}
for all compact Riemannian manifolds $(M^n,g)$, $n > 2k$, and all $u \in C^\infty(M)$;
see Section~\ref{sec:geometry} for a precise definition.
Examples include the scalar curvature, the $Q$-curvatures~\cite{MR1316845}, and the renormalized volume coefficients~\cites{MR2493186,MR1758076};
the latter specialize to the $\sigma_k$-curvatures when $k\leq2$ or the underlying manifold is locally conformally flat~\cite{MR2351380}.
The vector space of CVIs of weight $-2k$ is classified~\cite{MR3955546} when $k \leq 3$.

Given a CVI $I$ of weight $-2k$, there is is a minimal positive integer $r \leq 2k$, called the \defn{rank}, such that there is a formally self-adjoint, conformally covariant, polydifferential operator $D \colon \bigl( C^\infty(M) \bigr)^{\otimes(r-1)} \to C^\infty(M)$ for which
\begin{equation}
 \label{eqn:I-to-D}
 D^g(1^{\otimes (r-1)}) = \left( \frac{n-2k}{r} \right)^{r-1}I^g
\end{equation}
on any Riemannian manifold $(M^n,g)$, $n > 2k$~\cite{MR4392224}; see Section~\ref{sec:geometry} for details.
For example, the $Q$-curvatures have rank $2$, and the renormalized volume coefficient $v_k$ has maximal rank $2k$.
Equation~\eqref{eqn:I-to-D} implies that if $(M^n,g)$ is compact and $u \in C^\infty(M)$ is positive, then
\begin{equation}
 \label{eqn:total-I-to-D}
 \left( \frac{n-2k}{r} \right)^{r-1}\int_M I^{g_u} \dvol_{g_u} = \int_M uD^g(u^{\otimes(r-1)}) \dvol_g ,
\end{equation}
where $g_u := u^{\frac{2r}{n-2k}}g$.

Equation~\eqref{eqn:total-I-to-D} allows one to define a Yamabe-type constant associated with a CVI $I$ either in terms of the total $I$-curvature or the Dirichlet energy of the associated operator $D$.
However, for CVIs of rank at least three, the operator $u \mapsto D(u^{\otimes(r-1)})$ is nonlinear. Hence, one typically must constrain the allowable representatives of the conformal class to obtain a finite Yamabe-type constant and guarantee that the PDE is elliptic at its minimizers.
A \defn{geometric cone} of weight $-2k$ and rank $r$ is a function $U$ which assigns to each compact Riemannian manifold $(M^n,g)$, $n > 2k$, an open subset $U^g \subset C^\infty(M)$ such that
\begin{enumerate}
 \item if $u \in U$, then $cu \in U$ for each constant $c>0$;
 \item given $w \in C^\infty(M;\bR_+)$, it holds that $wU^g = U^{g_w}$ for $g_w := w^{\frac{2r}{n-2k}}g$; and
 \item if $\pi \colon \cM^n \to M^n$ is a finite connected covering, then $\pi^\ast(U^g) \subseteq U^{\pi^\ast g}$.
\end{enumerate}
Different choices of geometric cone $U$ may be relevant for the same CVI.
For example, when studying the fourth-order $Q$-curvature, one can consider the cones of all functions, of all positive functions, and of all positive functions which are conformal factors for a metric of positive scalar curvature~\cite{MR3509928}.

The \defn{$(I,U)$-Yamabe constant} is
\begin{equation}\label{eqn:yamabe}
 Y_{(I,U)}(M^n,[g]) := \left( \frac{r}{n-2k} \right)^{r-1} \inf_{u \in U^g} \left\{ \int_M uD^g(u^{\otimes(r-1)}) \dvol_g \suchthat \int_M \lv u \rv^{\frac{rn}{n-2k}} \dvol_g = 1 \right\} .
\end{equation}
When $U=C^\infty(M)$, we call this the \defn{$I$-Yamabe constant}, denoted $Y_I(M^n,[g])$.
The Yamabe constant~\cite{MR888880} is proportional to the $Q_2$-Yamabe constant.
The \defn{metric $(I,U)$-Yamabe constant} is
\begin{equation}\label{eqn:geometric-yamabe}
  Y_{(I,U)}^+(M^n,[g]) := \inf_{g_u \in [g]_U} \left\{ \int_M I^{g_u} \dvol_{g_u} \suchthat \Vol_{g_u}(M) = 1 \right\}.
\end{equation}
When $U=C^\infty(M)$, we call this the \defn{metric $I$-Yamabe constant}, denoted $Y_I^+(M^n,[g])$.
Here
\begin{equation}\label{eq:conformalclass}
  [g]_U:= \left\{ u^{\frac{2r}{n-2k}}g \suchthat u \in U^g \cap C^\infty(M;\bR_+) \right\}
\end{equation}
is the set of all Riemannian metrics conformal to $g$ with a conformal factor in $U$.
Note that
\begin{equation*}
    Y_{(I,U)}(M^n,[g]) \leq Y_{(I,U)}^+(M^n,[g]),
\end{equation*}
and if $U \subseteq C^\infty(M;\bR_+)$, then equality holds.
Additionally, if there is a positive minimizer of $Y_{(I,U)}(M^n,[g])$, then equality holds.

There is no general result that guarantees the existence of a minimizer for the $(I,U)$-Yamabe constant.
Instead, we isolate certain essential features of a ``good'' existence result.

\begin{definition}
    \label{defn:geometric-aubin-set}
    Let $I$ be a CVI of weight $-2k$ and rank $r$, and let $U$ be a geometric cone.
    A \defn{geometric Aubin set} for $(I,U)$ is a set $\mA$ of compact Riemannian $n$-manifolds, $n > 2k$, such that
    \begin{enumerate}
        \item if $(M^n,g) \in \mA$, then $0 < Y_{(I,U)}(M^n,[g]) \leq Y_{(I,U)}(S^n,[\grd])$ with equality if and only if $(M^n,g)$ is conformally equivalent to $(S^n,\grd)$;
        \item if $(M^n,g) \in \mA$, then there is a positive minimizer of the $(I,U)$-Yamabe constant~\eqref{eqn:yamabe}; and
        \item if $(M^n,g) \in \mA$ and $\pi \colon \cM^n \to M^n$ is a finite connected covering, then $(\cM^n,\pi^\ast g) \in \mA$. 
    \end{enumerate}
\end{definition}

Our terminology is inspired by Aubin's observation~\cite{MR0431287} that if $Y_{Q_2}(M^n,[g]) < Y_{Q_2}(S^n,[\grd])$, then minimizers of the Yamabe constant exist.
Results of Aubin~\cite{MR0431287} and Schoen~\cite{MR788292} imply that the set of all compact Riemannian $n$-manifolds, $n \geq 3$, with positive Yamabe constant is a geometric Aubin set for the pair $(I,U) = (Q_2,C^\infty(M))$.
A result of Gursky, Hang, and Lin~\cite{MR3509928} implies that the set of all compact Riemannian manifolds with positive $Q_2$- and $Q_4$-Yamabe constant is a geometric Aubin set for the pair $(I,U) = (Q_4,C^\infty(M))$.
In Section~\ref{sec:examples}, we list all maximal examples of geometric Aubin sets known to the authors.

Our definitions capture the key ingredients of the Bettiol--Piccione argument:

\begin{theorem}
 \label{main-thm}
 Let $I$ be a CVI of weight $-2k$ and rank $r$, and let $U$ be a geometric cone.
 Suppose that there is a nonempty geometric Aubin set $\mA$ for $(I,U)$.
 If $(M^n,g) \in \mA$ and $\pi_1(M)$ has infinite profinite completion, then for each $\ell \in \bN$ there is a finite regular covering $\pi \colon \cM^n \to M^n$ such that $(\cM^n,[\pi^\ast g])$ admits pairwise nonhomothetic representatives $\{ g_j \}_{j=1}^\ell$ with constant $I$-curvature.
\end{theorem}

Here we say that two Riemannian metrics $g_1,g_2$ on $M$ are \defn{homothetic} if there is a diffeomorphism $\Phi \in \Diff(M)$ and a constant $c>0$ such that $\Phi^\ast g_1 = c^2g_2$.
Theorem~\ref{main-thm} does not require the full strength of a geometric Aubin set;
it only requires that the set $\mA$ be closed under finite regular coverings and that there is a constant $C=C(\mA)>0$ such that if $(M^n,g) \in \mA$, then $Y_{(I,U)}(M^n,[g])>0$ and there is a representative $\hg \in [g]_U$ such that $\Vol_{\hg}(M)=1$ and $I^{\hg} \leq C$ is constant.

Theorem~\ref{main-thm} implies that there are infinitely many functions $u \in C^\infty(\cM;\bR_+)$ for which $I^{\cg_u}$ is constant in the universal cover $\cpi \colon \cM^n \to M^n$, where $\cg_u := u^{\frac{2r}{n-2k}}\cpi^\ast g$.
As in our companion paper~\cite{ACPW}, the Ferrand--Obata Theorem~\cites{MR1371767,MR0303464,MR1334876} implies we can in fact find infinitely many nonhomothetic solutions in the universal cover:

\begin{theorem}
 \label{infinite-main-thm}
 Let $I$ be a CVI of weight $-2k$ and rank $r$, and let $U$ be a geometric cone.
 Suppose that there is a nonempty geometric Aubin set $\mA$ for $(I,U)$.
 Suppose that $(M^n,g) \in \mA$ is such that $\pi_1(M)$ has infinite profinite completion.
 If the conformal universal cover of $(M^n,[g])$ is not conformal to Euclidean space, then there is a countable set $\{ \cg_j \}_{j\in \bN}$ of pairwise nonhomothetic periodic metrics $\cg_j \in [\cpi^\ast g]$, each of which has constant $I$-curvature.
\end{theorem}

Here the \defn{conformal universal cover} of $(M^n,[g])$ is the conformal manifold $(\cM^n,[\cpi^\ast g])$, where $\cpi \colon \cM^n \to M^n$ is the universal cover.
A metric $\cg \in [ \cpi^\ast g]$ is \defn{periodic} if it descends to a compact quotient of $\cM^n$.

We emphasize that Theorem~\ref{infinite-main-thm} is new even for the $Q_4$-curvature (cf.\ \cite{MR4251294}).
Compact conformal quotients of Euclidean space are classified~\cite{MR0987162}.
This allows one to replace the assumption on the conformal universal cover in Theorem~\ref{infinite-main-thm} by, for example, the assumption that $(M^n,g)$ has positive Yamabe constant.

We conclude this introduction with a brief comment about the differences between this paper and its companion~\cite{ACPW}.
The companion paper focuses entirely on the case of the scalar curvature, where Theorem~\ref{main-thm} recovers the aforementioned result of Bettiol and Piccione~\cite{MR3803113}*{Theorem~1.3} and where Theorem~\ref{infinite-main-thm} is new.
That paper also discusses the implications of these results for the singular Yamabe problem, including the best known estimates on the cardinality of the set of nonhomothetic solutions on $S^n \setminus S^k$.
There are many simplifications coming from the affirmative resolution of the Yamabe Problem~\cite{MR888880}.
In contrast, the present paper describes an axiomatic framework under which the Bettiol--Piccione argument can be applied.
Theorems~\ref{q-corollary}--\ref{vk-corollary} indicate the power of this framework, though additional work is needed to determine the full range of CVIs which admit nonempty geometric Aubin sets.

This article is organized as follows:

In Section~\ref{sec:geometry} we summarize the key definitions and facts involving CVIs and further discuss Yamabe-type constants.

In Section~\ref{sec:examples}, we list all maximal geometric Aubin sets known to the authors, including a new geometric Aubin set for the higher-order $Q$-curvatures.

In Section~\ref{sec:proof}, we recall a key property~\cite{MR3803113} of compact manifolds whose fundamental group has infinite profinite completion.
We then prove Theorems~\ref{main-thm} and \ref{infinite-main-thm}, and derive Theorems~\ref{q-corollary}, \ref{lcf-q-corollary}, and~\ref{vk-corollary} as corollaries.

In Appendix~\ref{app:aubin}, we prove a generalization of the Aubin Lemma~\cites{MR0433500,MR2301449}:
the $Q_{2k}$-Yamabe constant, when it is positive and admits positive minimizers, strictly increases when passing to a finite connected covering.
This fact is needed to produce geometric Aubin sets for the $Q_{2k}$-curvatures.

\section{Conformally variational Riemannian invariants}
\label{sec:geometry}
This section summarizes necessary facts and definitions involving conformally variational Riemannian invariants~\cites{MR4392224,MR3955546}.

A \defn{natural Riemannian scalar invariant} on $n$-manifolds is an assignment $I$ to each Riemannian manifold $(M^n,g)$ of a function $I^g \in C^\infty(M)$ by a universal linear combination $I$ of complete contractions of $g$, $g^{-1}$, and covariant derivatives of the Riemann curvature tensor.
Universality means that the coefficients do not depend on $(M,g)$.
We say that $I$ is \defn{homogeneous} of \defn{weight} $-2k$ if $I^{c^2g} = c^{-2k}I^g$ for all constants $c>0$.
Note that $k$ is necessarily a nonnegative integer.
When the metric $g$ is clear from context, we write $I$ for $I^g$.

A \defn{conformally variational invariant} (CVI) is a natural homogeneous scalar Riemannian invariant $I$ such that for all Riemannian manifolds $(M^n,g)$, the linear operator
\begin{equation*}
    C^\infty(M) \ni u \mapsto \left. \frac{\partial}{\partial t}\right|_{t=0} I^{e^{2tu}g} \in C^\infty(M)
\end{equation*}
is formally self-adjoint with respect to the $L^2$-inner product induced by the Riemannian volume element $\dvol_g$.
The formal self-adjointness of this operator is equivalent~\cite{MR2389992}*{Lemma~2} to the existence of a natural Riemannian functional $\mF \colon [g] \to \bR$ such that
\begin{equation*}
    \left. \frac{\mathrm{d}}{\mathrm{d}t}\right|_{t=0} \mF(e^{2tu}g) = \int_M uI^g \dvol_g
\end{equation*}
for all $u \in C^\infty(M)$.
When $n>2k$, one can take $\mF(g) = \frac{1}{n-2k}\int_M I^g\dvol_g$.

A \defn{natural $(r-1)$-differential operator}, $r \in \bN := \{1,2,3,\dotsc\}$, is an assignment $D$ to each Riemannian manifold $(M^n,g)$ of an operator
\begin{equation*}
    D^g \colon \left( C^\infty(M) \right)^{\otimes(r-1)} \to C^\infty(M) 
\end{equation*}
such that $D^g\left(u_1 \otimes \dotsm \otimes u_{r-1}\right)$ can be expressed as a universal linear combination of complete contractions of tensor products of $g$, $g^{-1}$, covariant derivatives of the Riemann curvature tensor, and covariant derivatives of the functions $u_j$.
The tensor product is over $\bR$;
thus $(r-1)$-differential operators are multilinear operators which are differential when all but one of the factors in the tensor product is held fixed.
Note that a natural $0$-differential operator is a natural scalar Riemannian invariant.
A \defn{natural polydifferential operator} is an operator which is a natural $(r-1)$-differential operator for some $r \in \mathbb{N}$.

We now restrict our attention to natural polydifferential operators that are homogeneous of degree $-2 k$ in $g$ on manifolds of dimension $n \geq 2k$.

A natural $(r-1)$-differential operator $D$ is \defn{conformally covariant} if  there is a multi-index $a \in \bR^{r-1}$ and a constant $b \in \bR$ such that
$$
D^{e^{2 \Upsilon} g}\left(u_1 \otimes \dotsm \otimes u_{r-1}\right)=e^{-b \Upsilon} D^g\left(e^{a_1 \Upsilon} u_1 \otimes \dotsm \otimes e^{a_{r-1}\Upsilon} u_{r-1}\right)
$$
for all Riemannian manifolds $\left(M^n, g\right)$ and all $\Upsilon, u_1, \ldots, u_{r-1} \in C^{\infty}(M)$.
In this case $D$ is homogeneous of degree $\lv a\rv - b := a_1 + \dotsm + a_{r-1} - b$, and we call $(a, b)$ the \defn{bidegree} of $D$.

A natural $(r-1)$-differential operator $D$ is \defn{formally self-adjoint} if for every Riemannian manifold $\left(M^n, g\right)$ and every $u_1, \ldots, u_{r} \in C^{\infty}(M)$ such that $u_1 \cdots u_{r}$ has compact support, the \defn{Dirichlet form}
$$
\kD^g\left(u_1 \otimes \dotsm \otimes u_{r}\right) := \int_M u_1 D^g\left(u_2 \otimes \dotsm \otimes u_{r}\right) \dvol_g
$$
is symmetric;
i.e.\ $\kD(u_{\sigma(1)} \otimes \dotsm \otimes u_{\sigma(r)}) = \kD(u_1\otimes \dotsm \otimes u_r)$ for all permutations $\sigma$ of $\{ 1, \dotsc, r\}$.
Note that if $D$ is formally self-adjoint, then the map $\left(u_1 \otimes \dotsm \otimes u_{r-1}\right) \mapsto D\left(u_1 \otimes \dotsm \otimes u_{r-1}\right)$ is symmetric.
If $D \colon \bigl( C^\infty(M) \bigr)^{\otimes{(r-1)}} \to C^\infty(M)$ is a formally self-adjoint, conformally covariant, polydifferential operator which is homogeneous of degree $-2k$, then it has bidegree $\bigl( \frac{n-2k}{r}, \frac{n(r-1)+2k)}{r} \bigr)$~\cite{MR4392224}*{Lemma~3.8}.
In particular, if additionally $n>2k$, then
\begin{equation*}
    \int_M u_1 D^{g_w}(u_2 \otimes \dotsm \otimes u_r) \dvol_{g_w} = \int_M wu_1 D^g (wu_2 \otimes \dotsm \otimes wu_{r}) \dvol_g ,
\end{equation*}
where $g_w := w^{\frac{2r}{n-2k}}g$.

Let $I$ be a CVI of weight $-2 k > -n$. A natural $(r-1)$-differential operator $D$ \defn{recovers} $I$ if for all Riemannian manifolds $(M^n,g)$, one has
    \begin{equation}\label{relationpolydifferentialCVI}
        D^g(1\otimes \dotsm \otimes1)=\left(\frac{n-2 k}{r}\right)^{r-1} I^g .
    \end{equation}
(There is a corresponding notion~\cite{MR4392224} for dimension $n=2k$ which is not needed in this paper.)
A natural $(r-1)$-differential operator is \defn{associated} to $I$ if it is formally self-adjoint, conformally covariant, and recovers $I$. 
Note that if $D$ recovers $I$, then Equation~\eqref{relationpolydifferentialCVI} is equivalent to
\begin{equation}\label{equivalence}
    \left(\frac{n-2 k}{r}\right)^{r-1} u^{\frac{n(r-1) + 2k}{n-2k}} I^{g_u}=D^g\bigl( u^{\otimes(r-1)} \bigr)
\end{equation}
for all $u \in C^\infty(M;\bR_+)$, where again $g_u=u^{\frac{2r}{n-2k}}g$.

Case, Lin, and Yuan~\cite{MR4392224}*{Theorem~1.6} showed that each CVI has an associated polydifferential operator. 
More precisely, if $I$ is a CVI of weight $-2 k$, then there is a minimal integer $1 \leq r \leq 2 k$ for which there exists a natural $(r-1)$-differential operator associated with $I$.
This number is the rank of $I$ defined in the introduction.

Let $I$ be a CVI of weight $-2k>-n$ with rank $r$.
The \defn{$I$-Yamabe quotient} of a compact Riemannian manifold $(M^n,g)$ is
\begin{equation}\label{yamabequotient}
	\mathcal{I}^{g}_{2k} := \frac{\int_M I_{2k}^{g}\dvol_g}{{\rm Vol}_{g}(M)^{\frac{n-2k}{n}}} .
\end{equation}
The transformation rule~\eqref{equivalence} implies that if $u \in C^\infty(M;\bR_+)$, then
\begin{equation*}
    \mathcal{I}^{g_u}_{2k} = \left(\frac{r}{n-2k}\right)^{r-1}\frac{\int_M u D^g\left(u^{\otimes(r-1)}\right) \dvol_{g}}{\left(\int_{M} u^{\frac{rn}{n-2k}} \dvol_g\right)^{\frac{n-2k}{n}}} .
\end{equation*}
If $U$ is a geometric cone, then the metric $(I,U)$-Yamabe constant~\eqref{eqn:geometric-yamabe} is
\begin{equation}\label{eq:Iyamabeinvariant}
    Y_{(I,U)}^+ = \inf \left\{ \mI_{2k}^{g_u} \suchthat g_u \in [g]_U \right\}.
\end{equation}
where we recall that $[g]_U$ is defined by Equation~\eqref{eq:conformalclass}.

\section{Examples of geometric Aubin sets}\label{sec:examples}

In this section, we list all known maximal geometric Aubin sets.
We also identify a geometric Aubin set for the higher-order $Q$-curvatures.
To that end, we denote by $\mathcal{M}^n$ the set of all compact Riemannian $n$-manifolds.

\subsection{Second-order Q-curvature}\label{sec:yamabe}
The second-order $Q$-curvature is proportional to the scalar curvature, the subject of the well-studied Yamabe Problem~\cite{MR888880}.
Since $Y_{Q_2}(M^n,[g])>0$ if and only if there exists a metric $\hg \in [g]$ with positive scalar curvature, the set
\begin{equation*}
    \mathcal{A} = \left\{ (M^n,g) \in \mathcal{M}^n \suchthat Y_{Q_2}(M^n,[g])>0 \right\}
\end{equation*}
is closed under finite connected coverings.
The resolution of the Yamabe Problem implies that $\mA$ is a geometric Aubin set for $(Q_2,C^\infty(M))$.

\subsection{Renormalized volume coefficients}\label{renormalized-volume-coefficients}
The renormalized volume coefficients $v_k$, defined on Riemannian manifolds of dimension $n \geq 2k$, are CVIs for which the equation $v_k^{g_u}=c$ is in general a fully nonlinear second-order PDE in the conformal factor~\cites{MR2493186,MR3955546}.
Solutions of this equation are known to exist when $k \in \{1,2\}$ or when restricted to locally conformally flat manifolds~\cites{MR2072215,MR2362323,MR2290138};
in these cases $v_k$ is the $\sigma_k$-curvature~\cites{MR1738176,MR2351380}.

We reformulate the aforementioned results in terms of a geometric Aubin set.
To that end, let
\begin{equation*}
    \Sch := \frac{1}{n-2}\left( \Ric - \frac{R}{2(n-1)}g \right)
\end{equation*}
be the Schouten tensor of $(M^n,g)$, regarded, using $g^{-1}$, as a section of $T^\ast M \otimes TM$.
Given $j \in \bN$, denote by $\sigma_j( \Sch)$ the $j$-th elementary symmetric function of the eigenvalues of $ \Sch$.
The \defn{positive $j$-cone} is the set
\begin{equation*}
    \Gamma_j^+ := \left\{ \hat{g} \in [g] \suchthat \sigma_1( \Sch), \sigma_2( \Sch), \dotsc, \sigma_j( \Sch) > 0 \right\} .
\end{equation*}
If $\hat g \in \Gamma_k^+$, then the equation $\sigma_k^{\hat g}=f$, expressed as a PDE in the conformal factor $\hg = e^{2u}g$, is elliptic~\cite{MR1738176}*{Section~6.3}.
Ellipticity is key in producing a geometric Aubin set:

\begin{proposition}
    \label{vk-aubin-set}
    Let $k \in \bN$.
    Pick an integer $n > 2k$.
    Then
    \begin{equation*}
        \mA := \left\{ (M^n,g) \in \mM^n \suchthat \Gamma_k^+ \not= \emptyset, W^g = 0 \right\}
    \end{equation*}
    is a geometric Aubin set for $(v_k,\Gamma_k^+)$, where $W^g$ is the Weyl tensor of $g$.
\end{proposition}

\begin{proof}
    Since $v_1 = Q_2$, the case $k=1$ follows from the discussion in Subsection~\ref{sec:yamabe}.

    Suppose that $k \geq 2$.
    Since the conditions $\Gamma_k^+\not=\emptyset$ and $W^g=0$ are equivalent to the existence of a metric $\hg \in [g]$ with $\sigma_j(\Sch^{\hg}) > 0$, $1 \leq j \leq k$, and $W^{\hg}=0$, we see that $\mA$ is closed under finite connected coverings.
    Let $(M^n,g) \in \mA$.
    Since $n \geq 5$ and the Weyl tensor vanishes, $(M^n,g)$ is locally conformally flat.
    Therefore $v_k$ is a nonzero multiple of the $\sigma_k$-curvature~\cite{MR2351380}*{Proposition~1}.
    A result of Guan and Wang~\cite{MR2072215}*{Theorem~1(A)} implies that $\mA$ is a geometric Aubin set.
\end{proof}

\subsection{Fourth-order Q-curvature}
The difficulty in studying the $\sigma_k$-curvatures comes from the need to restrict to an elliptic cone.
In contrast, the difficulty in studying the higher-order $Q$-curvatures~\cite{MR1316845} comes from the general lack of a maximum principle for higher-order operators.
The most general existence results are available in the fourth-order case~\cites{MR3420504,MR3518237,MR3509928}:

\begin{proposition}
    \label{q4-aubin-set}
    Fix an integer $n \geq 6$.
    Then
    \begin{equation*}
        \mA := \left\{ (M^n,g) \in \mM^n \suchthat Y_{Q_2}(M^n,[g]), Y_{Q_4}(M^n,[g]) > 0 \right\}
    \end{equation*}
    is a geometric Aubin set for $(Q_4,C^\infty(M))$.
\end{proposition}

\begin{remark}
    It is expected~\cite{MR3509928}*{p.\ 1352} that Proposition~\ref{q4-aubin-set} is true for all $n \geq 5$.
\end{remark}

\begin{proof}
    The condition $Y_{Q_2}(M^n,[g]), Y_{Q_4}(M^n,[g]) > 0$ is equivalent~\cite{MR3509928}*{Theorem~1.1} to the existence of a metric $\hg \in [g]$ such that $Q_2^{\hg},Q_4^{\hg}>0$.
    Thus, $\mA$ is closed under finite connected coverings.

    Let $(M^n,g) \in \mA$.
    As noted above, we may assume that $Q_2^g,Q_4^g > 0$.
    Thus the Green's function for the Paneitz operator is positive~\cite{MR3420504}*{Proposition~C}.
    Since $Y_{Q_2}(M^n,[g])>0$, we deduce that
    \begin{equation*}
        Y_{Q_4}(M^n,[g]) \leq Y_{Q_4}(S^n,[\grd])
    \end{equation*}
    with equality if and only if $(M^n,[g])$ is conformally diffeomorphic to $(S^n,[\grd])$~\cite{MR4330485}*{Theorem~1.3}.
    Combining an existence result of Mazumdar~\cite{MR3542966}*{Theorem~3} with the fact that $Y_{Q_4}(S^n,[\grd])$ is achieved by a positive function~\cite{Beckner1993}*{Theorem~6} implies that there is a positive minimizer of $Y_{Q_4}(M^n,[\grd])$.
\end{proof}

\subsection{Higher-order Q-curvatures}\label{higher-order-q}
The aforementioned result of Mazumdar states that if $(M^n,g)$, $n > 2k$, is compact Riemannian $n$-manifold for which the GJMS operator~\cite{MR1190438} $P_{2k}$ of order $2k$ has positive Green's function and satisfies $0 < Y_{Q_{2k}}(M^n,[g]) < Y_{Q_{2k}}(S^n,[\grd])$, then there is a positive minimizer of $Y_{Q_{2k}}(M^n,[g])$.
The condition $Y_{Q_{2k}}(M^n,[g])>0$ is not preserved under finite covers, nor is either condition easy to check.
The most general geometric sufficient conditions for the existence of minimizers follow from the work of Case and Malchiodi~\cite{case-malchiodi}:

Let $m \in \bN_0$ be a nonnegative integer.
An \defn{$m$-special Einstein product ($m$-SEP)}~\cite{MR2574315} is a Riemannian product manifold $(M_1^m \times M_2^{n-m}, g_1 \oplus g_2)$ such that $\Ric_{g_1} = -(m-1)\lambda g_1$ and $\Ric_{g_2} = (n-m-1)\lambda g_2$ for some constant $\lambda > 0$.
Case and Malchiodi showed that if $n$ is sufficiently large relative to $k$ and $m$, then the Green's function of $P_{2k}$ is positive and
\begin{equation*}
    0 < Y_{Q_{2k}}(M_1^m \times M_2^{n-m}, [g_1 \oplus g_2]) < Y_{Q_{2k}}(S^n,[\grd]) .
\end{equation*}
However, the property of being conformal to an $m$-SEP is not preserved under finite connected coverings.
Instead, we say that $(X^n,g)$ is \defn{virtually} conformal to an $m$-SEP if there is a finite connected covering $\pi \colon \cX^n \to X^n$ such that $(\cX^n,\pi^\ast g)$ is conformal to an $m$-SEP.
This gives rise to a geometric Aubin set:

\begin{proposition}\label{case-malchiodi}
    Fix $k,m \in \bN$.
    There is an integer $N = N(k,m)$ such that if $n \geq N$, then
    \begin{equation*}
        \mA := \left\{ (X^n,g) \in \mM^n \suchthat \text{$(X^n,g)$ is virtually conformal to an $m$-SEP} \, \right\}
    \end{equation*}
    is a geometric Aubin set for $(Q_{2k},C^\infty(M))$.
\end{proposition}

\begin{proof}
    Case and Malchiodi showed~\cite{case-malchiodi}*{Theorem~1.3} that there is an $N(k,m) \in \bN$ such that if $(X^n,g)$ is conformal to an $m$-SEP, then $0 < Y_{Q_{2k}}(X^n,[g]) < Y_{Q_{2k}}(S^n,[\grd])$ and there is a $\hg \in [g]$ which minimizes the $Q_{2k}$-Yamabe constant $Y_{Q_{2k}}(X^n,[g])$.
    Moreover, $N(k,m) \geq 2k+2m-1$.
    Set $N := N(k,m)$.

    Suppose that $\pi \colon \cX^n \to X^n$ is a finite connected covering and that $(\cX^n,\cg)$, $\cg := \pi^\ast g$, is conformal to an $m$-SEP.
    Then the GJMS operator $P_{2k}^{\cg}$ satisfies the Strong Maximum Principle~\cite{case-malchiodi}*{Theorem~1.2}.
    In particular, $\ker P_{2k}^{\cg} = \{0\}$.
    Therefore, $P_{2k}^g$ has a trivial kernel, and so its Green's function $G_{2k}^g$ exists.
    Since $P_{2k}^{\cg}(\pi^\ast G_{2k}^g) \geq 0$, we deduce that $G_{2k}^g>0$.
    Let $d$ denote the degree of $\pi$.
    On the one hand, since $\mI_{2k}^{\pi^\ast\hat g} = d^{2k/n}\mI_{2k}^{\hat g}$ for every $\hat{g} \in [g]$, we see that
    \begin{equation*}
        Y_{Q_{2k}}(X^n,[g]) = \inf \left\{ \mI_{2k}^{\hat{g}} \suchthat \hat{g} \in [g] \right\} \geq d^{-2k/n} Y_{Q_{2k}}(\cX^n,[\cg]) > 0 .
    \end{equation*}
    On the other hand, since there is a positive minimizer of $Y_{Q_{2k}}(\cX^n,[\cg])$, we deduce from Lemma~\ref{aubin-lemma} below that
    \begin{equation*}
        Y_{Q_{2k}}(X^n,[g]) < Y_{Q_{2k}}(\cX^n,[\cg]) < Y_{Q_{2k}}(S^n,[\grd]) .
    \end{equation*}
    Hence there is a $\hg \in [g]$ which minimizes $Y_{Q_{2k}}(X^n,[g])$~\cite{MR3542966}*{Theorem~3}.

    Finally, suppose that $(X^n,g)$ is virtually conformal to an $m$-SEP and that $\varphi \colon \cX^n \to X^n$ is a finite connected covering.
    Let $(M_1^m \times M_2^{n-m}, g_1 \oplus g_2)$ be an $m$-SEP and let $\pi \colon M_1^m \times M_2^{n-m} \to X^n$ be a finite connected covering such that $g_1 \oplus g_2 \in [\pi^\ast g]$.
    Since $(M_2^{n-m},g_2)$ has positive Ricci curvature, $\pi_1(M_2)$ is finite.
    By passing to a finite cover, we can assume that $\pi_1(M_1 \times M_2) = \pi_1(M_1)$.
    Since $\pi$ and $\varphi$ are covering maps, we may regard $H := \pi_\ast\bigl( \pi_1(M_1) \bigr) \cap \varphi_\ast\bigl( \pi_1(\cX) \bigr)$ as a (necessarily finite index) subgroup of $\pi_1(M_1)$.
    Let $\widehat{\pi} \colon \widehat{M}_1 \to M_1$ be a finite connected covering with $\widehat{\pi}_\ast\bigl(\pi_1(\widehat{M}_1)\bigr) = H$.
    Then $(\widehat{M}_1^m \times M_2^{n-m} , \widehat{\pi}^\ast g_1 \oplus g_2)$ is an $m$-SEP.
    Since $\pi_\ast(\widehat{\pi} \times 1)_\ast(\pi_1(\widehat{M}_1 \times M_2))$ is a subgroup of $\varphi_\ast\bigl(\pi_1(\cX)\bigr)$, the Lifting Criterion~\cite{MR2766102}*{Theorem~11.18} implies that $\pi \circ (\widehat{\pi} \times 1) \colon \widehat{M}_1^m \times M_2^{n-m} \to X^n$ lifts to a finite connected covering $\widehat{\varphi} \colon \widehat{M}_1^m \times M_2^{n-m} \to \cX^n$.
    Therefore, $(\cX,\varphi^\ast g)$ is virtually conformal to an $m$-SEP.
\end{proof}

\begin{remark}
 Proposition~\ref{case-malchiodi} does not give the best expected geometric Aubin set.
 For example, Andrade, Piccione, and Wei~\cite{arXiv:2306.00679} conjectured that if $(M^n,g)$, $n \geq 7$, is a compact Riemannian manifold with $Y_{Q_2}(M^n,[g]),Y_{Q_4}(M^n,[g])>0$ and which admits a metric $\hg \in [g]$ with $Q_{6}^{\hg} \geq 0$ and $Q_6^{\hg} \not=0$, then there is a positive minimizer for $Y_{Q_6}(M^n,[g])$.
 If this is true, then the set of all compact Riemannian $n$-manifolds which admit a conformal representative with positive $Q_2$-, $Q_4$-, and $Q_6$-curvature is a geometric Aubin set.
\end{remark}

While it is known that $N(k,m) \geq 2k + 2m - 1$ in Proposition~\ref{case-malchiodi}, the minimal value of $N(k,m)$ is not known.
However, results of Qing and Raske~\cite{MR2219215} give the optimal value of $N$ for locally conformally flat manifolds which are virtually conformal to an $m$-SEP.

\begin{proposition}
    \label{qing-raske}
    Fix $k,m \in \bN$.
    Let $n \geq 2k + 2m - 1$ be an integer.
    Then
    \begin{equation*}
        \mA := \left\{ (X^n,g) \in \mM^n \suchthat \text{$W^g=0$ and $(X^n,g)$ is virtually conformal to an $m$-SEP} \, \right\}
    \end{equation*}
    is a geometric Aubin set for $(Q_{2k},C^\infty(M))$.
\end{proposition}

\begin{proof}
    The case $k=1$ follows from the discussion of Subsection~\ref{sec:yamabe}.

    Suppose now that $k \geq 2$.
    Let $(X^n,g) \in \mA$ is conformal to an $m$-SEP.
    Then $n \geq 5$ and $(X^n,g)$ is locally conformally flat.
    Since $(X^n,g)$ is conformal to an $m$-SEP, we see that $0 < Y_{Q_{2k}}(X^n,[g]) < Y_{Q_{2k}}(S^n,[\grd])$~\cite{case-malchiodi}*{Theorem~1.1 and Lemma~4.3}.
    Hence there is a minimizer $\hg \in [g]$ for $Y_{Q_{2k}}(X^n,[g])$~\cite{MR3542966}*{Theorem 3}.
    The rest of the argument follows as in the proof of Proposition~\ref{case-malchiodi}.
\end{proof}

\section{Proofs of our main results}
\label{sec:proof}

Let $G$ be a group with identity $e_G$.
We say that $G$ is \defn{residually finite} if for every $g \in G$ with $g \not= e_G$, there is a finite group $H$ and a group homomorphism $\pi \colon G \to H$ such that $\pi(g) \not= e_H$.
The \defn{profinite completion} of a group $G$ is the limit
\begin{equation*}
 \hG := \varprojlim G / N ,
\end{equation*}
where $N$ ranges over all finite index normal subgroups of $G$.
It readily follows that if $G$ is infinite and residually finite, then its profinite completion is infinite.
In particular, the Selberg--Malcev Lemma~\cite{MR1299730}*{Section~7.6} implies that the fundamental group of a compact symmetric space of noncompact type---for example, a compact hyperbolic manifold---has infinite profinite completion.
Additional examples have been discussed by Bettiol and Piccione~\cite{MR3803113}*{Section~3.2}.

The basic idea of the proof of Theorem~\ref{main-thm} is to inductively apply the following topological observation of Bettiol and Piccione~\cite{MR3803113}*{Lemma~3.6} to construct a tower of finite regular coverings for which minimizers of the $(I,U)$-Yamabe constant must be nonhomothetic.

\begin{lemma}
 \label{bettiol-piccione}
 Let $(M^n,g)$ be a compact Riemannian manifold for which $\pi_1(M)$ has infinite profinite completion.
 Given $V \in \bR$, there exists a finite regular covering $\pi \colon \cM^n \to M^n$ such that $\Vol_{\pi^\ast g}(\cM) > V$.
\end{lemma}

Lifting these metrics to the conformal universal cover produces infinitely many representatives with constant $I$-curvature for which no two conformal factors, written in the form of Equation~\eqref{equivalence}, are constant multiples of one another.
The following result uses the Ferrand--Obata Theorem~\cites{MR1334876,MR1371767} on noncompact manifolds to conclude that these representatives are generically nonhomothetic;
this argument also appears in~\cite{ACPW} for the special case when $I$ is the scalar curvature.

\begin{lemma}
    \label{construction-lemma}
    Let $I$ be a CVI of weight $-2k$ and rank $r$, and let $U$ be a geometric cone.
    Suppose that there is a nonempty geometric Aubin set $\mA$ for $(I,U)$.
    Let $(M^n,g) \in \mA$ be such that there is a sequence $\{\pi_j \colon M_j^n \to M^n\}_{j \in \bN}$ of finite connected coverings of degree $m_j \geq j$ and a sequence $\{g_j\}_{j \in \bN}$ of minimizers of $Y_{(I,U)}^+(M_j^n,[\pi_j^\ast g])$ such that for each $j \in \bN$, there is a diffeomorphism $\Phi_j \in \Diff(\cM)$ and a constant $c_j > 0$ such that $\Phi_j^\ast \cpi^\ast g = c_j^2 \cpi_j^\ast g_j$, where $\cpi \colon \cM^n \to M^n$ and $\cpi_j \colon \cM^n \to M_j^n$ are the universal covers of $M^n$ and $M_j^n$, respectively.
    Then $(\cM^n,\cpi^\ast g)$ is conformally equivalent to flat Euclidean space.
\end{lemma}

\begin{proof}
    Since $(M^n,g) \in \mA$ and the statement of the lemma depends only on $[g]$, we may assume without loss of generality that $I^g = Y_{(I,U)}(M^n,[g])$ and $\Vol_g(M) = 1$.
    Since $m_j \to \infty$ as $j \to \infty$, we see that $\cM$ is noncompact.
    
    Fix $j \in \bN$.
    Observe that $\cpi = \pi_j \circ \cpi_j$.
    Let $\Aut(\cpi)$ and $\Aut(\cpi_j)$ denote the groups of deck transformations for $\cpi$ and $\cpi_j$, respectively.
    Let $F$ be a fundamental domain for $\cpi$.
    On the one hand, since $\Aut(\cpi_j)$ is a finite index subgroup of $\Aut(\cpi)$, there is a $\tau_j \in \Aut(\cpi)$ such that
    \begin{equation*}
        \Vol_{\tau_j^\ast \cpi_j^\ast g_j}(F) = \min \left\{ \Vol_{\sigma^\ast \cpi_j^\ast g_j}(F) \suchthat \sigma \in \Aut(\cpi) \right\} .
    \end{equation*}
    Therefore $\Vol_{\tau_j^\ast \cpi_j^\ast g_j}(F) \leq m_j^{-1}$.
    On the other hand, we may pick a $\sigma_j \in \Aut(\cpi)$ such that $\Psi_j := \sigma_j \circ \Phi_j \circ \tau_j$ satisfies $\Psi_j(F) \cap F \not= \emptyset$.
    Then $\Psi_j^{-1}(F) \cap F \not= \emptyset$.
    Suppose to the contrary that $(\cM,\cpi^\ast g)$ is not conformally equivalent to Euclidean space.
    The Ferrand Obata--Theorem~\cite{MR1371767}*{Theorem~A${}_1$} yields a $\Psi \in \Conf(\cM,\cpi^\ast g)$ such that $\Psi_j \to \Psi$ in the compact-open topology.

    We now derive a contradiction by computing
    \begin{equation*}
        \mI_j := \frac{\int_F I^{\Psi_j^\ast\cpi^\ast g} \dvol_{\Psi_j^\ast \cpi^\ast g}}{\Vol_{\Psi_j^\ast\cpi^\ast g}(F)^{\frac{n-2k}{n}}}
    \end{equation*}
    in two ways.
    First, since $I^g = Y_{(I,U)}(M^n,[g])$, we directly compute that
    \begin{equation}
        \label{eqn:mI-via-Psi}
        \mI_j = Y_{(I,U)}(M^n,[g])\Vol_{\cpi^\ast g}(\Psi_j(F))^{\frac{2k}{n}} .
    \end{equation}
    Second, note that $\Psi_j^\ast\cpi^\ast g = c_j^2 \tau_j^\ast \cpi_j^\ast g_j$.
    Since $I$ is homogeneous of degree $-2k$ with respect to constant rescalings and since $I^{g_j}=Y_{(I,U)}(M_j^n,[\pi_j^\ast g])$, we compute that
    \begin{equation*}
        \mI_j = \frac{\int_F I^{\tau_j^\ast\cpi_j^\ast g_j} \dvol_{\tau_j^\ast\cpi_j^\ast g_j}}{\Vol_{\tau_j^\ast\cpi_j^\ast g_j}(F)^{\frac{n-2k}{n}}} = Y_{(I,U)}(M_j^n,[\pi_j^\ast g])\Vol_{\tau_j^\ast\cpi_j^\ast g_j}(F)^{\frac{2k}{n}} .
    \end{equation*}
    Our choice of $\tau_j$ and the nonnegativity of $Y_{(I,U)}(M_j^n,[\pi_j^\ast g])$ imply that
    \begin{equation*}
        \mI_j \leq Y_{(I,U)}(M_j^n,[\pi_j^\ast g])m_j^{-\frac{2k}{n}} .
    \end{equation*}
    Therefore $\limsup_{j\to\infty} \mI_j \leq 0$.
    Since $Y_{(I,U)}(M^n,[g])>0$, we deduce from Equation~\eqref{eqn:mI-via-Psi} that $\Vol(\cpi^\ast g)(\Psi_j(F)) \to 0$ as $j \to \infty$, a contradiction.
\end{proof}

Lemmas~\ref{bettiol-piccione} and~\ref{construction-lemma} allow us to adapt an argument of Bettiol and Piccione~\cite{MR3803113}*{Theorem~1.3} to prove the following generalization of Theorems~\ref{main-thm} and~\ref{infinite-main-thm}.

\begin{theorem}
    \label{precise-main-thm}
    Let $I$ be a CVI of weight $-2k$ and rank $r$, and let $U$ be a geometric cone.
    Suppose that there is a nonempty geometric Aubin set $\mA$ for $(I,U)$.
    Let $(M^n,g) \in \mA$ be such that $\pi_1(M)$ has infinite profinite completion.
    Then there is an infinite tower
    \begin{equation*}
        \dotsm \overset{\pi_{k+1}}{\longrightarrow} M_k^n \overset{\pi_{k}}{\longrightarrow} \dotsm \overset{\pi_3}{\longrightarrow} M_2^n \overset{\pi_2}{\longrightarrow} M_1^n \overset{\pi_1}{\longrightarrow} M_0^n := M^n
    \end{equation*}
    of finite regular coverings and a sequence $(g_j)_{j=0}^\infty$ of minimizers of $Y_{(I,U)}^+(M_j^n,[\Pi_j^\ast g])$ such that for each integer $j \geq 0$, the set $\{ (\Pi_j^\ell)^\ast g_\ell \}_{\ell=0}^j$ consists of pairwise nonhomothetic representatives of $[\Pi_j^\ast g]$ with constant $I$-curvature, where $\Pi_j^\ell := \pi_{\ell+1} \circ \dotsm \circ \pi_j \colon M_j^n \to M_\ell^n$ and $\Pi_j := \Pi_j^0 \colon M_j^n \to M^n$ are defined for all $j \geq \ell$, with the convention $\Pi_j^j = \Id$.
    Moreover, if $(\cM^n,\cpi^\ast g)$ is not conformally equivalent to $(\bR^n,\mathrm{d}x^2)$, then, after passing to a subtower if necessary, the metrics $\{ \cpi_j^\ast g_j \}_{j=0}^\infty$ are pairwise nonhomothetic, where $\cpi \colon \cM^n \to M^n$ and $\cpi_j \colon \cM^n \to M_j^n$ are the universal covers of $M^n$ and $M_j^n$, respectively.
\end{theorem}

\begin{proof}
 To simplify notation, we denote $Y := Y_{(I,U)}(M^n,[g])$.
 Similarly, given a finite connected covering $\Pi_j \colon M_j^n \to M^n$, we denote $Y_j := Y_{(I,U)}(M_j^n,[\Pi^\ast g])$.
 Since $\mA$ is a geometric Aubin set, $Y_j = Y_{(I,U)}^+(M_j^n,[\Pi^\ast g])>0$.

 Suppose that a finite regular covering $\Pi_j \colon M_j^n \to M^n$ is given.
 Since $\mA$ is a geometric Aubin set, we can pick a minimizer $g_j \in [\Pi_j^\ast g]$ of $Y_j$.
 Let $V_j > 0$ be such that
 \begin{equation}
     \label{eqn:pick-Vj}
     Y_jV_j^{\frac{2k}{n}} > Y_{(I,U)}(S^n,[\grd]) .
 \end{equation}
 Choose, using Lemma~\ref{bettiol-piccione}, a finite regular covering $\pi_{j+1} \colon M_{j+1}^n \to M_j^n$ such that
 \begin{equation*}
     \Vol_{\pi_{j+1}^\ast g_j}(M_{j+1}^n) > V_j .
 \end{equation*}
 Direct computation yields
 \begin{equation}
    \label{eqn:construction-inequality}
     \mI^{\pi_{j+1}^\ast g_j} = Y_j \Vol_{\pi_{j+1}^\ast g_j}(M_{j+1}^n)^{\frac{2k}{n}} > Y_{(I,U)}(S^n,[g_{\mathrm{rd}}]) \geq \mI^{g_j} .
 \end{equation}
 Denote $\Pi_{j+1} := \Pi_j \circ \pi_{j+1} \colon M_j^n \to M^n$.
 Then $\pi_{j+1}^\ast g_j$ is not a minimizer of $Y_{j+1}$.

 The above construction yields an infinite tower $\{ \pi_j \colon M_j^n \to M_{j-1}^n \}_{j \in \bN}$ of finite regular coverings and a sequence $\{ g_j \}_{j \in \bN_0}$ of minimizers of $Y_j$.
 Moreover, Inequality~\eqref{eqn:construction-inequality} implies that
 \begin{equation*}
     \mI^{(\Pi_j^{\ell-1})^\ast g_{\ell-1}} = \mI^{(\Pi_j^\ell)^\ast \pi_\ell^\ast g_{\ell-1}} = \mI^{\pi_\ell^\ast g_{\ell-1}}\left( \deg \Pi_j^\ell \right)^{\frac{2k}{n}} > \mI^{g_\ell} \left( \deg \Pi_j^\ell \right)^{\frac{2k}{n}} = \mI^{(\Pi_j^\ell)^\ast g_\ell}
 \end{equation*}
 for each $\ell \in \bN$.
 In particular, $\mI^{g_j} < \mI^{(\Pi_j^{j-1})^\ast g_{j-1}} < \dotsm < \mI^{(\Pi_j^0)^\ast g_0}$.
 The scale and diffeomorphism invariance of the total $I$-curvature implies that these metrics are pairwise nonhomothetic.

 Finally, suppose that $(\cM^n,\cpi^\ast g)$ is not conformally equivalent to Euclidean space.
 Lemma~\ref{construction-lemma} implies that no subsequence of $\{\cpi_j^\ast g_j \}_{j=0}^\infty$ consists of pairwise homothetic metrics.
 The conclusion readily follows.
\end{proof}

The existence of compact Riemannian manifolds admitting many nonhomothetic metrics of constant $Q_{2k}$-curvature in their conformal class follows from work of Qing and Raske~\cite{MR2219215} and of Case and Malchiodi~\cite{case-malchiodi}.

\begin{proof}[Proof of Theorem~\ref{q-corollary}]
 Let $U$ be the geometric cone of all smooth functions.
 Let $N$ and $\mA$ be as in Proposition~\ref{case-malchiodi}.
 Then $(M_1^m \times M_2^{n-m}, g_1 \oplus g_2) \in \mA$ and $\pi_1(M_1 \times M_2)$ has infinite profinite completion.
 The conclusions now follow from Theorems~\ref{main-thm} and~\ref{infinite-main-thm}.
\end{proof}

\begin{proof}[Proof of Theorem~\ref{lcf-q-corollary}]
 Let $U$ be the geometric cone of all smooth functions.
 Let $\mA$ be as in Proposition~\ref{qing-raske}.
 Then $(M^m \times S^{n-m}, \ghyp \oplus \grd) \in \mA$ and $\pi_1(M)$ has infinite profinite completion.
 The conclusions now follow from Theorem~\ref{main-thm} and~\ref{infinite-main-thm}.
\end{proof}

The existence of compact Riemannian manifolds admitting many nonhomothetic metrics of constant $v_{k}$-curvature in their conformal class follows from the solution~\cites{MR2362323,MR2072215} of the $\sigma_k$-Yamabe Problem.

\begin{proof}[Proof of Theorem~\ref{vk-corollary}]
 Fix $k,m \in \bN$.
 Given $n \in \bN$, the $k$-th elementary symmetric function of the eigenvalues of the block diagonal matrix $A = -I_m \oplus I_{n-m}$, where $I_m$ and $I_{n-m}$ are the $m \times m$ and $(n-m) \times (n-m)$ identity matrices, respectively, is
 \begin{equation}
  \label{eqn:sigmak-formula}
  \sigma_k(A) = \sum_{j=0}^k (-1)^j\binom{m}{j}\binom{n-m}{k-j} .
 \end{equation}
 In particular, $\sigma_k(A) = n^k/k! + \mathcal{O}(n^{k-1})$ for $n \gg 1$.
 It follows that there is a constant $N = N(k,m)$ such that if $n \geq N$, then $\sigma_j(A) > 0$ for all $j \in \{ 1, 2, \dotsc, k \}$.
 
 Now let $(M_1^m,\ghyp)$ and $(M_2^{n-m},\grd)$ be compact spaceforms with constant sectional curvature $-1$ and $1$, respectively, where $n \geq N$.
 Then their Riemannian product $(M_1^m \times M_2^{n-m}, \ghyp \oplus \grd)$ is locally conformally flat~\cite{MR2371700}*{Example~1.167(3)}.
 Moreover, the Schouten tensor of the product is
 \begin{equation}
  \label{eqn:product-schouten}
  P = \frac{1}{2}\left( -\ghyp \oplus \grd \right) .
 \end{equation}
 Thus $(M_1^m \times M_2^{n-m}, \ghyp \oplus \grd)$ is in the geometric Aubin set of Proposition~\ref{vk-aubin-set}.
 The conclusions now follow from Theorem~\ref{main-thm} and~\ref{infinite-main-thm} and the Selberg--Malcev Lemma.
\end{proof}

\begin{remark}
 \label{rk:sigmak-N-values}
 In the notation of the proof of Theorem~\ref{vk-corollary}, it is easily checked that $N(1,m) = 2m+1$.
 Direct computation using Equation~\eqref{eqn:sigmak-formula} implies that $N(n,1) = n-2k+1$.
 The dimensions $m$ and $n$ for which $v_2$ or $v_3$ vanish on $(H^m \times S^{n-m})$ are known~\cite{MR3959563}*{Lemma~6.1}.
 Thus estimates for $N(2,m)$ and $N(3,m)$ are known.
 Indeed, Equation~\eqref{eqn:sigmak-formula} implies that
 \begin{align*}
  2v_1 & = n-2m , \\
  4v_2 & = \frac{n^2-(4m+1)n + 4m^2}{2} , \\
  8v_3 & = \frac{(n-2m)(n^2-(4m+3)n+4m^2+2)}{6} .
 \end{align*}
 In particular, one has
 \begin{align*}
  N(2,m) & > \frac{4m+1 + \sqrt{8m+1}}{2} , \\
  N(3,m) & > \frac{4m+3 + \sqrt{24m + 1}}{2}.
 \end{align*}
 When $k \geq 4$, there are only finitely many choices of $m$ and $n$ for which $v_k$ vanishes~\cite{MR4166038}, so sharp estimates for $N(k,m)$, $k \geq 4$, are more difficult to find.
\end{remark}

\appendix
\section{The Aubin Lemma for the Q-curvature}
\label{app:aubin}

The Aubin Lemma~\cite{MR0433500}*{Theorem 6} asserts that if $\pi \colon (\cM^n,\cg) \to (M^n,g)$ is a finite connected Riemannian covering and $Y_{Q_2}(\cM^n,[\cg])>0$, then $Y_{Q_2}(M^n,[g]) < Y_{Q_2}(\cM^n,[\cg])$.
Aubin's proof assumes a regular covering, but this assumption can be removed~\cite{MR2301449}*{Lemma~3.6}.
Both proofs require only the existence of a minimizer for $Y_{Q_2}(\cM^n,[\cg])$ and Jensen's inequality.
As such, they can be extended to the higher-order $Q$-curvatures:
	
\begin{lemma}
	\label{aubin-lemma}
	Let $k \in \bN$.
    Let $(M^n,g)$, $n > 2k$, be a compact Riemannian manifold.
	Suppose that $\pi \colon \cM^n \to M^n$ is a finite connected covering of degree at least two such that $Y_{Q_{2k}}(\cM^n,[\pi^\ast g]) > 0$ and there is a positive minimizer for $Y_{Q_{2k}}(\cM^n,[\pi^\ast g])$.
	Then
	\begin{equation*}
		Y_{Q_{2k}}(M^n,g) < Y_{Q_{2k}}(\cM^n,[\pi^\ast g]) .
	\end{equation*}
\end{lemma}

\begin{proof}
    Denote $\cg := \pi^\ast g$.
    Let $u \in C^\infty(\cM;\bR_+)$ be such that
    \begin{align*}
        P_{2k}^{\cg}u & = \frac{n-2k}{2}Y_{Q_{2k}}(\cM^n,[\cg]) u^{\frac{n+2k}{n-2k}} , \\
        \int_{\cM} u^{\frac{2n}{n-2k}}\dvol_{\cg} & = 1 ,
    \end{align*}
    where $P_{2k}$ is the GJMS operator~\cite{MR1190438} of order $2k$.

    Denote by $d\geq2$ the degree of $\pi$.
    Given $x \in \cM$, set $\{ x_1, \dotsc, x_d \} = \pi^{-1}\bigl(\pi(x)\bigr)$.
    Given a constant $p > 0$, define
    \begin{equation*}
        \cu_{(p)}(x) := \sum_{j=1}^d u(x_j)^p .
    \end{equation*}
    This defines a function $\cu_{(p)} \in C^\infty(\cM;\bR_+)$.
    Set $\cu:=\cu_{(1)}$.
    Observe that if $V \subset M$ is an evenly covered open set, then
    \begin{align*}
        \cu_{(p)} = \sum_{j=1}^d (u \circ \gamma_j)^p ,
    \end{align*}
    where $\gamma_j \colon V_1 \to V_j$ are isometries between the connected components $V_1, \dotsc, V_d$ of $\pi^{-1}(V)$.
    We may piece these isometries together to produce $m$ distinct isometries
    \begin{equation*}
        \widetilde{\gamma}_j \colon \cM \setminus \mS \to \cM \setminus \mS, \qquad j \in \{ 1, \dotsc, d \} ,
    \end{equation*}
    where $\mS \subset \cM$ is a piecewise smooth compact $(n-1)$-dimensional submanifold of $\cM^n$ such that $\pi^{-1}\bigl(\pi(\mS)\bigr) = \mS$.
    Therefore
    \begin{align*}
        \cu_{(p)}  = \sum_{j=1}^d \cu_j^p ,
    \end{align*}
    where $\cu_j := u \circ \widetilde{\gamma}_j$.
    Since each $\widetilde{\gamma}_j$ is an isometry, we deduce that
    \begin{align*}
        P_{2k}^{\cg}\cu & = \frac{n-2k}{2}Y_{Q_{2k}}(\cM^n,[\cg])\sum_{j=1}^d \cu_j^{\frac{n+2k}{n-2k}} , \\
        \int_{\cM} \cu_{\bigl(\frac{2n}{n-2k}\bigr)} \dvol_{\cg} & = d .
    \end{align*}

    Now, since $\cu(x) = \cu(x^\prime)$ for any $x,x^\prime \in \cM$ such that $\pi(x)=\pi(x^\prime)$, there is a $u_0 \in C^\infty(M;\bR_+)$ such that $\cu = u_0 \circ \pi$.
    Direct computation gives
    \begin{equation}
        \label{eqn:test-function}
        \begin{split}
            \mI_{2k}^{u_0^{\frac{4}{n-2k}}g} & = d^{-\frac{2k}{n}}\mI_{2k}^{\cu^{\frac{4}{n-2k}}\cg} \\
        & = d^{-\frac{2k}{n}}\left( \int_{\cM} \cu P_{2k}^{\cg}(\cu)\dvol_{\cg} \right)\left( \int_{\cM} \cu^{\frac{2n}{n-2k}}\dvol_{\cg}\right)^{-\frac{n-2k}{n}} .
        \end{split}
    \end{equation}
    Since $d \geq 2$ and $u>0$, applying H\"older's inequality and then convexity yields
    \begin{align*}
        \cu \left( \sum_{j=1}^d \cu_j^{\frac{n+2k}{n-2k}} \right) & \leq \cu \left( \sum_{j=1}^d \cu_j^{\frac{2n}{n-2k}} \right)^{\frac{2k}{n}} \left( \sum_{j=1}^d \cu_j^{\frac{n}{n-2k}} \right)^{\frac{n-2k}{n}} < \cu^2\left( \sum_{j=1}^d \cu_j^{\frac{2n}{n-2k}} \right)^{\frac{2k}{n}} .
    \end{align*}
    Since $Y_{Q_{2k}}(\cM^n,[\cg])>0$, we deduce that
    \begin{align*}
        \int_{\cM} \cu P_{2k}^{\cg}(\cu) \dvol_{\cg} & = Y_{Q_{2k}}(\cM^n,[\cg])\int_{\cM} \cu \left( \sum_{j=1}^d \cu_j^{\frac{n+2k}{n-2k}} \right) \dvol_{\cg} \\
        & < Y_{Q_{2k}}(\cM^n,[\cg])\int_{\cM} \cu^2 \cu_{\bigl(\frac{2n}{n-2k}\bigr)}^{\frac{2k}{n}} \dvol_{\cg} \\
        & \leq Y_{Q_{2k}}(\cM^n,[\cg])\left( \int_{\cM} \cu^{\frac{2n}{n-2k}}\dvol_{\cg} \right)^{\frac{n-2k}{n}} \left( \int_{\cM} \cu_{\bigl(\frac{2n}{n-2k}\bigr)} \dvol_{\cg} \right)^{\frac{2k}{n}} \\
        & = d^{\frac{2k}{n}}Y_{Q_{2k}}(\cM^n,[\cg])\left( \int_{\cM} \cu^{\frac{2n}{n-2k}}\dvol_{\cg}\right)^{\frac{n-2k}{n}} .
    \end{align*}
    Combining this with Equation~\eqref{eqn:test-function} yields
    \begin{equation*}
        Y_{Q_{2k}}(M^n,[g]) \leq \mI_{2k}^{u_0^{\frac{4}{n-2k}}g} < Y_{Q_{2k}}(\cM^n,[\cg]) . \qedhere
    \end{equation*}

\section*{Declarations}

\subsection*{Funding}
This work was partially supported by Fundação de Amparo \`a Pesquisa do Estado de São Paulo (FAPESP), Conselho Nacional de Desenvolvimento Científico e Tecnológico (CNPq), Natural Sciences and Engineering Research Council of Canada (NSERC), Natural Sciences Foundation (NSF), Hong Kong Special Administrative Region General Research Fund (HKSAR CRF), and Simons Foundation. 
J.H.A. was supported by FAPESP \#2020/07566-3, \#2021/15139-0, and \#2023/15567-8 and CNPq \#409764/2023-0, \#443594/2023-6, \#441922/2023-6, and \#306014/2025-4. 
J.S.C. was partially supported by a Simons Foundation Collaboration Grant for Mathematicians and by the National Science Foundation under Award \#DMS-2505606.
P.P. was supported by FAPESP \#2022/16097-2, and CNPq \#313773/2021-1 and \#441922/2023-6. 
J.W. was supported by NSERC \#RGPIN-2018-03773 and  HKSAR-GRF General Research Grant \#14309824.

\subsection*{Conflict of interest}
The authors have no relevant financial or non-financial interests to disclose.

\subsection*{Data availability}
Data sharing not applicable as no datasets were generated or analyzed during the study.

\subsection*{Ethics approval}
Not applicable.
\end{proof}

\bibliography{references}
\bibliographystyle{abbrv}

\end{document}